\documentclass[12pt]{amsart}
 \usepackage{amscd,amsmath,amsthm,amssymb}
 \def\NZQ{\Bbb}               % the font for N,Z,Q,R,C
 \def\NN{{\NZQ N}}
 
 \def\ZZ{{\NZQ Z}}

 \def\DD{{\NZQ D}}
 \def\FF{{\NZQ F}}
 \def\GG{{\NZQ G}}
 \def\HH{{\NZQ H}}

 \def\frk{\frak}               % font for "Fraktur"

 \def\mm{{\frk m}}
 
 \def\nn{{\frk n}}
 \def\ab{{\bold a}}
 
 \def\xb{{\bold x}}

 \def\db{{\bold d}}

 \def\opn#1#2{\def#1{\operatorname{#2}}} % to make operators
 \opn\chara{char} \opn\length{\ell} \opn\pd{pd} \opn\rk{rk}
 \opn\projdim{proj\,dim} \opn\injdim{inj\,dim} \opn\rank{rank}
 \opn\depth{depth} \opn\grade{grade} \opn\height{height}
 \opn\embdim{emb\,dim} \opn\codim{codim}
 
 \opn\Tr{Tr} \opn\bigrank{big\,rank}
 \opn\superheight{superheight}\opn\lcm{lcm}
 \opn\trdeg{tr\,deg}%\emph{
 \opn\reg{reg} \opn\lreg{lreg} \opn\ini{in} \opn\lpd{lpd}
 \opn\size{size} \opn\sdepth{sdepth}
 \opn\link{link}\opn\fdepth{fdepth}\opn\lex{lex}
 \opn\div{div} \opn\Div{Div} \opn\cl{cl} \opn\Cl{Cl}
 \opn\Spec{Spec} \opn\Supp{Supp} \opn\supp{supp} \opn\Sing{Sing}
 \opn\Ass{Ass} \opn\Min{Min}\opn\Mon{Mon}
 \opn\Ann{Ann} \opn\Rad{Rad} \opn\Soc{Soc}
 \opn\Im{Im} \opn\Ker{Ker} \opn\Coker{Coker} \opn\Am{Am}
 \opn\Hom{Hom} \opn\Tor{Tor} \opn\Ext{Ext} \opn\End{End}
 \opn\Aut{Aut} \opn\id{id}
 
 \opn\nat{nat}
 \opn\pff{pf}%   \pf exists already
 \opn\Pf{Pf} \opn\GL{GL} \opn\SL{SL} \opn\mod{mod} \opn\ord{ord}
 \opn\Gin{Gin} \opn\Hilb{Hilb}\opn\sort{sort}
 \opn\Tot{Tot}
 \opn\aff{aff} \opn
 \con{conv} \opn\relint{relint} \opn\st{st}
 \opn\lk{lk} \opn\cn{cn} \opn\core{core} \opn\vol{vol}
 \opn\link{link} \opn\star{star}\opn\lex{lex}\opn\set{set}
 \opn\gr{gr}
 
 \def\pot#1#2{#1[\kern-0.28ex[#2]\kern-0.28ex]}

 \opn\dirlim{\underrightarrow{\lim}}
 \opn\inivlim{\underleftarrow{\lim}}
 \let\union=\cup

 \let\tensor=\otimes
 \let\iso=\cong
 
 \let\Sect=\bigcap
 \let\Dirsum=\bigoplus
 \let\Tensor=\bigotimes
 \let\to=\rightarrow
 \let\To=\longrightarrow
 \def\Implies{\ifmmode\Longrightarrow \else
         \unskip${}\Longrightarrow{}$\ignorespaces\fi}
 \def\implies{\ifmmode\Rightarrow \else
         \unskip${}\Rightarrow{}$\ignorespaces\fi}
 \def\iff{\ifmmode\Longleftrightarrow \else
         \unskip${}\Longleftrightarrow{}$\ignorespaces\fi}

 \let\:=\colon
 \newtheorem{Theorem}{Theorem}[section]
 \newtheorem{Lemma}[Theorem]{Lemma}
 \newtheorem{Corollary}[Theorem]{Corollary}

 \newtheorem{Examples}[Theorem]{Examples}
 \newtheorem{Definition}[Theorem]{Definition}

 \let\epsilon\varepsilon
 \let\kappa=\varkappa
 \def\qed{\ifhmode\textqed\fi
       \ifmmode\ifinner\quad\qedsymbol\else\dispqed\fi\fi}
 \def\textqed{\unskip\nobreak\penalty50
        \hskip2em\hbox{}\nobreak\hfil\qedsymbol
        \parfillskip=0pt \finalhyphendemerits=0}
 \def\dispqed{\rlap{\qquad\qedsymbol}}
 
 \opn\dis{dis}
 \def\pnt{{\raise0.5mm\hbox{\large\bf.}}}
 
 \opn\Lex{Lex}

\begin{document}

 \title{Generalized mixed product ideals}

 \author {J\"urgen Herzog,  Roya Moghimipor and Siamak Yassemi}

\address{J\"urgen Herzog, Fachbereich Mathematik, Universit\"at Duisburg-Essen, Campus Essen, 45117
Essen, Germany} \email{juergen.herzog@uni-essen.de}

\address{Roya Moghimipor, Department of Mathematics, Science and Research Branch of Islamic Azad University, Tehran, Iran} \email{roya\_moghimipour@yahoo.com}

\address{Siamak Yassemi, School of Mathematics, Statistics and Computer Science, College of
Science, University of Tehran, Tehran, Iran, and School of Mathematics, Institute
for Research in Fundamental Sciences (IPM), P.O. Box 19395-5746, Tehran, Iran.}\email{yassemi@ipm.ir}

 \begin{abstract}
We consider classes of ideals  which generalize the mixed product ideals introduced by Restuccia and Villarreal \cite{RV}, and also generalize the  expansion construction by Bayati and the first author \cite{BH}. We compute the minimal graded free resolution of  generalized mixed product ideals and show that the regularity of a generalized mixed product ideal coincides with regularity of the monomial ideal by  which it is induced.
 \end{abstract}

\thanks{The second author wants to thank the University of Duisburg-Essen for its hospitality during the preparation of this work.}

 \maketitle

\section*{Introduction}
In 2001 Restuccia and Villarreal \cite{RV} introduced mixed product ideals, which form  a particular class of squarefree monomial ideals, and they classified those among these ideals  which are normal, thereby generalizing results on normality of previously known cases. Subsequently, Rinaldo \cite{R} and Ionescu and Rinaldo \cite{IR} studied other algebraic and homological properties of this class of ideals, and Hoa and Tam \cite{HT} computed the regularity and some other algebraic invariants of mixed products of arbitrary graded ideals.

Mixed product ideals, as introduced by Restuccia and Villarreal are of the form $(I_qJ_r + I_pJ_s)S$, where for integers $a$ and $b$, the ideal  $I_a$  (resp.\ $J_b$) is the ideal generated by all squarefree monomials of degree $a$ in the polynomial ring $A = K[x_1,\ldots, x_n]$
(resp.\ of degree $b$ in the polynomial ring $B = K[y_1, \ldots, y_m]$),  and where $0 < p < q \leq n$, $0 < r < s \leq m$
and $S = K[x_1, \ldots,  x_n, y_1, \ldots, y_m]$. Thus the ideal $L=(I_qJ_r + I_pJ_s)S$ is obtained from the monomial ideal $I=(x^qy^r,x^py^s)$ by replacing $x^q$ by $I_q$, $x^p$ by $I_p$, $y^r$ by $J_r$ and $y^s$ by $J_s$. One may wonder whether the ideals $I$ and $L$ share any algebraic properties. Indeed they do: they have the same regularity. This observation prompted us to do a similar construction starting with  an arbitrary monomial ideal $I$ in the polynomial ring  $K[x_1,\ldots,x_n]$ over the field $K$. For this construction  we choose for each $i$ a set of new variables $x_{i1},x_{i2},\ldots,x_{im_i}$ and   replace in  each minimal generator $x_1^{a_1}x_2^{a_2}\cdots x_n^{a_n}$ of $I$ each of the  factor $x_i^{a_i}$ by a monomial ideal in $T_i=K[x_{i1},x_{i2},\ldots,x_{im_i}]$ generated in degree $a_i$.  We call the monomial ideal $L$ obtained in this way a {\em  generalized mixed product ideal} induced by $I$. This construction is not completely new. Indeed in the paper \cite{BH}  by Bayati and the first author a similar construction, called expansion, is made. There however, each  $x_i^{a_i}$ is replaced by $(x_{i1},\cdots, x_{im_i})^{a_i}$, while in our generalized mixed product ideals each  $x_i^{a_i}$ is replaced by  an {\em arbitrary} monomial ideal of $T_i$ generated in degree $a_i$. Thus this new class of monomial ideals  widely generalizes the classical mixed product ideals as well as the ideals obtained as expansions. Path ideals of complete bipartite graphs are examples of generalized mixed product ideals as introduced here.

The main result (Theorem~\ref{regularity}) of this paper states that a generalized mixed product ideal $L$ induced by $I$ has the same regularity as $I$, provided the ideals which replace the pure powers $x_i^{a_i}$ all have a linear resolution.  As a consequence  we obtain the result that under the above assumptions, $L$ has a linear resolution if and only if $I$ has a linear resolution, see Corollary~\ref{linear}. We also show that  the projective dimension of $L$ can be expressed in terms of the multi-graded shifts in the resolution of $I$ and the projective dimension of the ideals which replace the pure powers.

As in the paper \cite{BH} we construct a double complex whose total complex provides a multi-graded free resolution of the generalized mixed product ideal. The details of the construction are bit technical though the principle idea behind it is quite natural: the data of the multi-graded free resolution $\FF$ of $I$ are used to first construct an acyclic complex $\FF^*$ of direct sums of ideals with $H_0(\FF^*)=T/L$.  In the next step the free resolutions of the modules $F_i^*$ are patched together to form the desired double complex.

\section{Complexes attached to generalized mixed product ideals}
Let $K$ be a field and let $S=K[x_1,\ldots,x_n]$ be the polynomial ring over $K$ in the variables $x_1,\ldots,x_n$, and let $I\subset S$ be a monomial ideal with $I\neq S$ whose minimal set of generators  is $G(I)=\{\xb^{\ab_1},\ldots, \xb^{\ab_m}\}$. Here $\xb^{\ab}=x_1^{\ab(1)}x_2^{\ab(2)}\cdots x_n^{\ab(n)}$ for $\ab=(\ab(1),\ldots,\ab(n))\in\NN^n$.

Next we consider the polynomial ring $T$ over $K$ in the variables
\[
x_{11},\ldots,x_{1m_1},x_{21},\ldots,x_{2m_2},\ldots,x_{n1},\ldots,x_{nm_n}.
\]
Notice that $T=T_1\tensor_KT_2\tensor_K\cdots\tensor_K T_n$, where $T_j=K[x_{j1},x_{j2},\ldots,x_{jm_j}]$ for $j=1,\ldots,n$.

For $i=1,\ldots,n$ and $j=1,\ldots,m$ let $L_{i, \ab_j(i)}$ be a monomial ideal in the variables $x_{i1},x_{i2},\ldots,x_{im_i}$ such that
\begin{eqnarray}
\label{inclusion}
L_{i, \ab_j(i)}\subset L_{i, \ab_k(i)} \quad \text{whenever} \quad \ab_j(i)\geq \ab_k(i).
\end{eqnarray}

Given these ideals we define for $j=1,\ldots,m$ the monomial ideals
\begin{eqnarray}
\label{lj}
L_j=\prod_{i=1}^nL_{i, \ab_j(i)} \subset T,
\end{eqnarray}
and set $L=\sum_{j=1}^mL_j$.
We call $L$ a {\em generalized mixed product ideal} induced by $I$.

\begin{Examples}
\label{tania}
{\em
(a) Consider the mixed product ideals introduced by Restuccia and Villarreal \cite{RV}. A mixed product ideal is a monomial ideal of the form $L=I_kJ_r+I_sJ_t$ where the ideals $I_k$ and $I_s$ are monomial ideals generated in degree $k$  and $s$, respectively,  in the variables $x_1,\ldots,x_n$ and  $J_r$ and $J_t$ are monomial ideals generated in degree $r$ and $t$ in the variables $y_1,\ldots,y_m$. In our terminology $L$ is induced by the ideal $I=(x^ky^r,x^sy^t)$.

(b) As mentioned in \cite{RV} mixed product ideals also appear as generalized graph ideals (called path ideals by Conca and De Negri \cite{CN}) of complete bipartite graphs. Let $G$ a finite simple graph with vertices $x_1,\ldots,x_n$. A {\em path} of length $t$ in $G$ is sequence $x_{i_1},\ldots,x_{i_t}$ of pairwise distinct vertices such that $\{x_{i_k},x_{i_{k+1}}\}$ is an edge of  $G$. Let $K$ be a field and $S=K[x_1,\ldots,x_n]$ be the polynomial ring over $K$ in the variables $x_1,\ldots,x_n$. Then the  {\em path ideal} $I_t(G)$ is the ideal generated by all monomials $x_{i_1}\cdots x_{i_t}$ such that  $x_{i_1},\ldots,x_{i_t}$ is a path of length $t$.

Now let $G$ be a complete  $n$-partite graph with vertex set $V=V_1\union V_2\union \cdots\union V_n$ and $V_i=\{x_{i1},\ldots,x_{im_i}\}$ for $i=1,\ldots,n$. For this graph we have
\[
I_{t}(G)=\sum_{0\leq j_{i}\leq \min\{(t+1)/2,m_{i}\}, \sum_{i=1}^nj_{i}=t}I_{1j_{1}}I_{2j_{2}}\dots I_{nj_{n}}
\]
where the ideals $I_{ij_{i}}$ are the  monomial ideals generated by all squarefree monomials of  degree $j_{i}$ in the variables $\{x_{i1},\ldots,x_{im_i}\}$.
Thus  $I_{t}(G)$ is induced by the ideal $I$ of Veronese type  generated by the monomials $x_{1}^{j_{1}}x_{2}^{j_{2}}\dots x_{n}^{j_{n}}$ with $\sum_{i=1}^nj_{i}=t$ and $0\leq j_{i}\leq \min\{(t+1)/2,m_{i}\}$.
}
\end{Examples}

\medskip
Let
\begin{eqnarray}
\label{res}
\FF\: 0\to F_p\to F_{p-1} \to \cdots \to F_2\to F_1\to F_0\to S/I\to 0
\end{eqnarray}
be the $\ZZ^n$-graded minimal free $S$-resolution of $S/I$. Based on the data of this resolution, we construct an acyclic complex $\FF^*$ whose $0$th homology gives us $T/L$.

Let $F_i=\Dirsum_{j=1}^{\beta_i}S(-\ab_{ij})$ with $\ab_{ij}\in \NN^n$ for $i=1,\ldots,n$. Then $F_i=\Dirsum_{j=1}^{\beta_i}Sf_{ij}$ where $f_{ij}$ is a basis element of the free $S$-module $F_i$ of $\ZZ^n$-degree $\ab_{ij}$. Let $\partial$ denotes the chain map of $\FF$. Then
\begin{eqnarray}
\label{def}
\partial(f_{ij})=\sum_k\lambda^{(i)}_{kj}\xb^{\ab_{ij}-\ab_{i-1,k}}f_{i-1,k}.
\end{eqnarray}
Here $\lambda^{(i)}_{kj}=0$ if $\ab_{ij}=\ab_{i-1,k}$ or $\ab_{ij}-\ab_{i-1,k}\not\in \NN^n$. The matrices $(\lambda^{(i)}_{kj})_{k=1,\ldots,\beta_{i-1}\atop
j=1,\ldots,\beta_{i}}$ are called the {\em scalar matrices} of the resolution $\FF$.

\medskip
Now we choose for each of the generators $\xb^{\ab_j}$ of $I$ a monomial ideal $L_j$ in $T$ (not necessary of the form
(\ref{lj})), and define the following complex $\FF^*$: we set $F^*_0=T$ and $F^*_i=\Dirsum_{j=1}^{\beta_i} L_{ij}$ where the  monomial ideals $L_{ij}$ are inductively defined as follows: we let $L_{1j}=L_j$ for all $j$. Suppose $L_{i-1,j}$ is already defined for all $j$. For a given number $j$ with $1\leq j\leq \beta_i$, let $k_1,k_2,\ldots,k_r$ be the numbers for which $\lambda^{(i)}_{k_tj}\neq 0$. Then we set
\begin{eqnarray}
\label{intersection}
L_{ij}=\Sect_{t=1}^rL_{i-1, k_t}.
\end{eqnarray}
The chain map $\partial^*$ of $\FF^*$ is given by
\[
\partial^*\: \Dirsum_{j=1}^{\beta_i}L_{ij}\To \Dirsum_{j=1}^{\beta_{i-1}}L_{i-1,j},\quad u\mapsto \lambda^{(i)}u,
\]
where
\[
u= \left( \begin{array}{c}
u_1 \\
u_2 \\
\vdots\\
u_{\beta_i}
\end{array} \right)
\quad \text{with} \quad u_j\in L_{ij}.
\]
Observe that $\partial^*$ is well-defined, that is, $\partial^*( \Dirsum_{j=1}^{\beta_i}L_{ij})\subset \Dirsum_{j=1}^{\beta_{i-1}}L_{i-1,j}$. Indeed, let $v\in  \Dirsum_{j=1}^{\beta_i}L_{ij}$ be a column vector. We may assume that  $v_\ell=0$ for $\ell\neq j$. Then
\[
\partial^*(v)=\left( \begin{array}{c}
u_1 \\
u_2 \\
\vdots\\
u_{\beta_{i-1}}
\end{array} \right),
\]
where  $u_k=\lambda^{(i)}_{kj}v_j$ for $k=1,\ldots,\beta_{i-1}$. Thus it follows from (\ref{intersection}) that   $\partial^*(v)\in \Dirsum_{j=1}^{\beta_{i-1}}L_{i-1,j}$.

\begin{Lemma}
\label{easyforroya}
$\FF^*$  is a complex of $T$-modules with $H_0(\FF^*)=T/L$.
\end{Lemma}

\begin{proof}
We use the fact, shown in the next lemma, that $\lambda^{(i-1)} \lambda^{(i)}=0$ for all $i=2,\ldots, p$.
\[
\partial^*(\partial^*(u))=\lambda^{(i-1)}(\lambda^{(i)}(u)) =   (\lambda^{(i-1)} \lambda^{(i)})(u)=0\quad \text{for all} \quad u\in F_i^*.
\]
This shows that $\FF^*$ is a complex.

Since $\lambda^{(1)}=(1,1,\ldots 1)$, it follows that
\[
\partial^*(\Dirsum_{j=1}^{\beta_1}L_{1j}) =\sum_{j=1}^{\beta_1}L_{1j}= \sum_{j=1}^{m}L_j=L,
\]
and hence $H_0(\FF^*)=T/L$.
\end{proof}

We call $\FF^*$ the complex of $L_1,\ldots,L_m$ induced by $I$.

We attach to $\FF$ a complex $\bar{\FF}$ of $K$-vector spaces,  which we call the {\em scalar complex} of $\FF$.

\begin{Lemma}
\label{scalar} The following sequence of $K$-linear maps
\begin{eqnarray*}
\bar{\FF}\:\; 0\to K^{\beta_p}\overset{\lambda^{(p)}}{\to}K^{\beta_{p-1}}\overset{\lambda^{(p-1)}}{\to}\cdots \overset{\lambda^{(3)}}\to K^{\beta_{2}}\overset{\lambda^{(2)}}{\to} K^{\beta_1}\overset{\lambda^{(1)}}{\to} K\to  0
\end{eqnarray*}
is an exact complex.
\end{Lemma}

\begin{proof}
Let $J=(x_1-1,x_2-1,\ldots,x_n-1)$. Then $\bar{\FF}\iso \FF\tensor S/J$, and hence $\bar{\FF}$ is a complex with $H_i(\bar{\FF})\iso \Tor_i(S/I,S/J)$ for all $i$. It follows that each $H_i(\bar{\FF})$ is annihilated by $I+J$.   Since $I+J=S$, it follows that  $H_i(\bar{\FF})=0$ for all $i$, as desired.
\end{proof}

We now show that $\FF^*$ is acyclic for the choice of the ideals $L_j$ as given in (\ref{lj}).

\begin{Theorem}
\label{main}
Let $I\subset S=K[x_1,\ldots,x_n]$ be a monomial ideal with $G(I)=\{\xb^{\ab_1},\ldots,\xb^{\ab_m}\}$. For   $i=1,\ldots,n$ and $j=1,\ldots,m$ let $L_{i, \ab_j(i)}$ be a monomial ideal in the variables $x_{i1},x_{i2},\ldots,x_{im_i}$ satisfying  condition {\em (\ref{inclusion})},  and let  $\FF^*$ be the complex of $L_1,\ldots,L_m$ induced by $I$ with  $L_j=\prod_{i=1}^nL_{i, \ab_j(i)}$. Then $\FF^*$ is acyclic with $H_0(\FF^*)=T/L$, where $L=\sum_{j=1}^mL_j$.
\end{Theorem}

For the proof of this theorem we shall need:

\begin{Lemma}
\label{shifts}
Let $\FF$ be the $\ZZ^n$-graded minimal free $S$-resolution of $S/I$ as described in {\em (\ref{res})}. Given $1< i\leq p$ and $1\leq j\leq \beta_i$, let $k_1,k_2,\ldots,k_r$ be the integers for which $\lambda^{(i)}_{k_tj}\neq 0$. Then
$
\xb^{\ab_{ij}}=\lcm(\xb^{\ab_{i-1,k_1}}, \ldots, \xb^{\ab_{i-1,k_r}}).
$
\end{Lemma}

\begin{proof}
We have $\partial(f_{ij})=\sum_{t=1}^r\lambda^{(i)}_{k_tj}\xb^{\ab_{ij}-\ab_{i-1,k_t}}f_{i-1,k_t}$. Therefore, $\xb^{\ab_{ij}-\ab_{i-1,k_t}}\in S$ and this implies $\xb^{\ab_{i-1,k_t}}$ divides $\xb^{\ab_{ij}}$, and hence $\lcm(\xb^{\ab_{i-1,k_1}}, \ldots, \xb^{\ab_{i-1,k_r}})$ divides $\xb^{\ab_{ij}}$. Let
\[
\xb^{\db}=\xb^{\ab_{ij}}/\lcm(\xb^{\ab_{i-1,k_1}}, \ldots, \xb^{\ab_{i-1,k_r}}),
\]
and suppose that $\xb^{\db}\neq 1$. Since $g=\sum_{t=1}^r\lambda^{(i)}_{k_tj}(\xb^{\ab_{ij}-\ab_{i-1,k_t}}/\xb^{\db})f_{i-1,k_t}$ belongs to $\Ker \partial$, there exists $f\in F_i$ with $\partial(f)=g$. It follows that $\partial(\xb^{\db}f-f_{ij})=0$. This implies that $\xb^{\db}f-f_{ij}\in \Ker \partial\subset \mm F_i$, since $\FF$ is minimal. Since $f_{ij}\not\in \mm F_i$, but $\xb^{\db}f\in \mm F_i$, it follows that $\xb^{\db}f-f_{ij}\not\in \mm F_i$, a   contradiction.
\end{proof}

As  a consequence of this lemma one easily obtains by induction on $i$ the following result.

\begin{Corollary}
\label{easily}
For all integers $1\leq i\leq p$, $1\leq j\leq\beta_i$ and $1\leq l\leq n$ there exists an integer  $1\leq k\leq \beta_1$ such that $\ab_{ij}(l)=a_k(l)$. In particular, $L_{l,\ab_{ij}(l)}=L_{l,\ab_{k}(l)}$ is a monomial ideal in the variables $x_{l1},x_{l2},\ldots,x_{lm_l}$ and
\begin{eqnarray}
\label{contained}
L_{l,\ab_{i_1j_1}(l)}\subset L_{l,\ab_{i_2j_2}(l)}\quad \text{if}\quad \ab_{i_1j_1}(l)\geq \ab_{i_2j_2}(l).
\end{eqnarray}
\end{Corollary}

\begin{Corollary}
\label{lij}
For the choice of the ideals $L_j=\prod_{i=1}^nL_{i, \ab_j(i)}$ as given in Theorem~\ref{main},  the ideals $L_{ij}$ defined in {\em (\ref{intersection})} are presented as
\[
L_{ij}  =\prod_{l=1}^nL_{l, \ab_{ij}(l)}.
\]
\end{Corollary}

\begin{proof}
We prove the assertion by induction on $i$. For $i=1$, we have $L_{1j}=L_j=\prod_{i=1}^nL_{i, \ab_j(i)}=\prod_{i=1}^nL_{i, \ab_{1j}(i)}$. Now let $i>1$, and assume that the statement holds for $i-1$. By definition (\ref{intersection}) and our induction hypothesis we have
\[
L_{ij}  =\Sect_{t=1}^r L_{i-1,k_t}= \Sect_{t=1}^r \prod_{l=1}^nL_{l, \ab_{i-1,k_t}(l)}.
\]
By Corollary~\ref{easily} it follows  that $L_{l, \ab_{i-1,k_t}}(l)$ is a monomial ideal in the variables  $x_{l1},x_{l2},\ldots,x_{lm_l}$. Therefore, $\prod_{l=1}^nL_{l, \ab_{i-1,k_t}(l)}=\Sect_{t=1}^rL_{l, \ab_{i-1,k_t}}(l)$. Hence
\[
L_{ij}= \Sect_{t=1}^r\Sect_{l=1}^n L_{l, \ab_{i-1,k_t}(l)}=\Sect_{l=1}^n\Sect_{t=1}^rL_{l, \ab_{i-1,k_t}(l)}=\Sect_{l=1}^n L_{l, \ab_{ij}(l)}=\prod_{l=1}^nL_{l, \ab_{ij}(l)}.
\]
The third equation follows from (\ref{contained}) and the fact that $\ab_{ij}(l)=\max\{\ab_{i-1,k_1}(l),\ldots,\ab_{i-1,k_t}(l)\}$, see Lemma~\ref{shifts}.
\end{proof}

We need one more lemma  before we can give the proof of Theorem~\ref{main}.

\begin{Lemma}
\label{needed}
Let $\FF$ be the $\ZZ^n$-graded minimal free $S$-resolution of $S/I$ as described in {\em (\ref{res})}, and let
\[
\mu= \left( \begin{array}{c}
\mu_1 \\
\mu_2 \\
\vdots\\
\mu_{\beta_i}
\end{array} \right)
\in \Ker \lambda^{(i)}.
\]
Let $1\leq k_1<k_2<\cdots <k_r\leq \beta_i$ be the integers with $\mu_{i_{k_t}}\neq 0$ for $t=1,\ldots,r$.

Then there exists
\[
\rho=\left( \begin{array}{c}
\rho_1 \\
\rho_2 \\
\vdots\\
\rho_{\beta_{i+1}}
\end{array} \right)
\]
with the property  that $\lambda^{(i+1)}\rho=\mu$ and $\xb^{\ab_{i+1,j}}|\lcm(\xb^{\ab_{ik_1}},\ldots, \xb^{\ab_{ik_r}})$ for all $j$ with $\rho_j\neq 0$.
\end{Lemma}

\begin{proof}
Let $\xb^\db=\lcm(\xb^{\ab_{ik_1}},\ldots, \xb^{\ab_{ik_r}})$. Then $\mu \xb^\db\in \Ker((F_i)_\db\to (F_{i-1})_\db)$. Since the ${\db}$th component $\FF_\db$ of $\FF$ is exact, there exists $\rho\in K^{\beta_{i+1}}$ such that $\partial(\rho \xb^{\db})=\mu\xb^{\db}$. It follows that $\lambda^{(i+1)}\rho=\mu$. Since $\rho\xb^{\db}=\sum_{j=1}^{\beta_{i+1}}\rho_j\xb^{\db-\ab_{i+1,j}}f_{i+1,j}$, it follows that $\xb^{\ab_{i+1,j}}$ divides $\xb^\db$ for all $j$ with $\rho_j\neq 0$, as desired.
\end{proof}

\begin{proof}[Proof of Theorem~\ref{main}]
Since the complex $\FF^*$ is a complex of $\ZZ^q$-graded $T$-module, where $q$ is the number of variables of $T$ and each $F_i^*$ is a direct sum of monomial ideals of $T$ it suffices to show that if $g$ is a monomial of $T$ and $\mu$ an element of $K^{\beta_i}$ such that $\mu g\in \Ker(F_i^*\to F_{i-1}^*)$,  then there  exists $\rho\in K^{\beta_{i+1}}$ such that $\rho g\in F_{i+1}^*$ with $\partial^*(\rho g)=\mu g$.

We have $0=\partial^*(\mu g)=(\lambda^{(i)}\mu)g$. Therefore $\lambda^{(i)}\mu=0$. We now choose $\rho$ as in Lemma~\ref{needed}. We claim that $g\in L_{i+1,j}$ for all $j$ with $\rho_j\neq 0$, then $\rho g\in F_{i+1}^*$, and since $\lambda^{(i+1)}\rho=\mu$ it then follows that $\partial^*(\rho g)=\mu g$. Thus the theorem follows once we have proved the claim.

Let $1\leq k_1<k_2<\cdots <k_r\leq \beta_i$ be the integers with $\mu_{k_t}\neq 0$ for $t=1,\ldots,r$. Then $g\in L_{i,k_t}$ for $t=1\ldots,r$. We fix $j$ with $\rho_j\neq 0$. Then, by using Corollary~\ref{lij} and Lemma~\ref{needed} and Corollary~\ref{easily}  we see that
\begin{eqnarray*}
g\in \Sect_{t=1}^rL_{i,k_t}&=&\Sect_{t=1}^r\prod_{l=1}^nL_{l,\ab_{i,k_t(l)}}=\Sect_{t=1}^r\Sect_{l=1}^nL_{l,\ab_{i,k_t(l)}}\\
&=&\Sect_{l=1}^n\Sect_{t=1}^rL_{l,\ab_{i,k_t(l)}}\subset \Sect_{l=1}^nL_{l, \ab_{i+1,j}(l)}=\prod_{l=1}^nL_{l, \ab_{i+1,j}(l)}=L_{i+1,j}.
\end{eqnarray*}
\end{proof}

\section{On the regularity of generalized mixed ideals}

Let $I\subset S=K[x_1,\ldots,x_n]$ be a monomial ideal as in Section 1 with $G(I)=\{\xb^{\ab_1},\ldots, \xb^{\ab_m}\}$, and $L$ be defined as in $(\ref{lj})$.
In this section we will always assume that for $l=1,\dots,n$ and $j=1,\ldots,m$ the ideals  $L_{l,\ab_j(l)}$  have an $\ab_j(l)$-linear resolution.

The main result of this section is the following:

\begin{Theorem}
\label{regularity}
With the notation and the assumptions introduced, we have $$\reg I=\reg L.$$
\end{Theorem}

\cite[Theorem 2.8]{IR} of Ionescu and Rinaldo turns to be out a very special case of this theorem and is obtained from the next corollary in the case that both summands of $L$  have only two factors.

\begin{Corollary}
\label{giancarlo}
Let $L=I_1I_2\cdots I_n+J_1J_2\cdots J_n$ with $I_k$ and $J_k$ in $K[x_{k_1},\ldots,x_{km_k}]$ monomial ideals with linear resolution. Suppose that for $k=1,\ldots,n$ the ideal $I_k$ has a $d_k$-linear resolution and $J_k$ a $\delta_k$-linear resolution. Assume further that $I_k\subset J_k$ if $d_k\geq \delta_k$ and that $J_k\subset I_k$ if $\delta_k\geq d_k$.  Then
\[
\reg L=\sum_{k=1}^n\max\{d_k,\delta_k\}-1.
\]
\end{Corollary}

\begin{proof}
The ideal $L$ is the generalized mixed ideal induced by the ideal $$I=(x_1^{d_1}x_2^{d_2}\cdots x_n^{d_n} , x_1^{\delta_1}x_2^{\delta_2}\cdots x_n^{\delta_n}).$$ Since $\reg I=\sum_{k=1}^n\max\{d_k,\delta_k\}-1$, the assertion follows from Theorem~\ref{regularity}.
\end{proof}

\begin{Corollary}
\label{linear}
The following conditions are equivalent:
\begin{enumerate}
\item[(a)] $L$ has a linear resolution;
\item[(b)] $I$ has a linear resolution.
\end{enumerate}
\end{Corollary}

\begin{proof}
Because of Theorem~\ref{regularity} it suffices to show that $L$ is generated in degree $d$ if and only if $I$ is generated in degree $d$. To show this we use the fact that $\partial^*(F_2^*)\subset \nn F_1^*$ where $\nn$ is the graded maximal ideal of $T$, see Lemma~\ref{akihiro}. This then implies $\Dirsum_j L_j/\nn L_j\iso L/\nn L$. Our assumptions on the ideals $L_{i,\ab_{j(i)}}$ imply that $L_j$ is minimally generated in degree $|\ab_j|$. Hence  it follows that $L$ has generators exactly in the same degrees as $I$. Thus the desired conclusion follows.
\end{proof}

For the proof of Theorem~\ref{regularity} we need the following lemma.

\begin{Lemma}
\label{akihiro}
Let $\nn$  be the graded maximal ideal of $T$. Then $\partial^*(F_i^*)\subset \nn F_{i-1}^*$ for all $i>0$.
\end{Lemma}

\begin{proof}
For $i=1$, the assertion is obvious because $L\subset \nn$. Now let $i>1$. It is enough to show that $\partial^*(L_{ij})\subset \nn F_{i-1}^*$ for all $j$. Let $v\in L_{ij}$, $v\neq 0$. Then our assumption on the $L_{i,\ab_j(i)}$ together with Corollary~\ref{easily} and Corollary~\ref{lij} imply that $\deg v\geq |\ab_{ij}|$, where  for any $\ab \in \ZZ^n$ we denote by $|\ab|$ the sum $\sum_{l=1}^n \ab_{i}(l)$. Moreover, by the definition of $\partial^*$  it follows that  the component of $\partial^*(v)$ in $L_{i-1,k}$ is $0$, if $|\ab_{i-1,k}|\geq |\ab_{ij}|$ and is equal to $\lambda^{(i)}_{kj}v$ if $|\ab_{i-1,k}|<  |\ab_{ij}|$. Thus $\lambda^{(i)}_{kj}\deg v>|\ab_{i-1,k}|$ whenever $\lambda^{(i)}_{kj}\neq 0$.
\end{proof}

For the proof of Theorem~\ref{regularity} we use the strategy applied in the paper \cite{BH} and first construct a minimal $\ZZ^N$-graded resolution of $L$  where $N=\sum_{i=1}^nm_i$ (which is the number of variables of $T$). The resolution which we are going to construct will be the total complex of a certain double complex. The construction of this double complex is the following: For each $l=1,\ldots,n$ and each $j=1,\ldots,m$ we choose a minimal  $\ZZ^{m_l}$-graded free $T_l$-resolution $\HH^{(l,\ab_j(l))}$ of $L_{l,\ab_j(l)}$ with $\HH^{(l,\ab_{j_1}(l))}=\HH^{(l,\ab_{j_2}(l))}$  if $\ab_{j_1}(l)=\ab_{j_2}(l)$. We use these complexes to construct $\ZZ^N$-graded resolutions of the ideals $L_{ij}$, and let
\begin{eqnarray}
\label{product}
\GG^{(ij)}=\Tensor_{l=1}^n\HH^{(l,\ab_{ij}(l))}.
\end{eqnarray}
Here $\HH^{(l,\ab_{ij}(l))}$ is the complex $\HH^{(l,\ab_k(l))}$ if  $\ab_{ij}(l)=\ab_k(l)$, see Corollary~\ref{easily}.

As in \cite[Proposition 3.4]{BH} it follows that $\GG^{(ij)}$ is a minimal  $\ZZ^n$-graded free $T$-resolution of $L_{ij}$.

Let $\GG^{(i)}= \Dirsum_{j=1}^{\beta_i}\GG^{(ij)}$. Then $\GG^{(i)}$ is a minimal $\ZZ^N$-graded free $T$-resolution of $F_i^*=\Dirsum_{j=1}^{\beta_i}L_{ij}$.

It remains to define complex homomorphisms $\sigma_i\: \GG^{(i)}\to \GG^{(i-1)}$  for $i=1,\ldots p$ which extend the chain maps $\partial^*\: F_i^*\to F_{i-1}^*$ and such that $\sigma_{i-1}\circ \sigma_i=0$  and $\sigma_i(\GG^{(i)})\subset \nn \GG^{(i-1)}$ for all $i$. To define $\sigma_i$ it is enough to describe its  components
\[
\sigma_i^{(kj)}\:\ \GG^{(ij)}= \Tensor_{l=1}^n\HH^{(l,\ab_{ij}(l))}\to \GG^{(i-1,k)}= \Tensor_{l=1}^n\HH^{(l,\ab_{i-1,k}(l))}.
\]
We let  $\sigma_i^{(kj)}=\lambda_{kj}^{(i)}\tau_{i}^{(1,kj)}\tensor \tau_{i}^{(2,kj)}\tensor \cdots \tensor \tau_{i}^{(n,kj)}$, where
\[
\tau_i^{(l,kj)}\: \HH^{(l,\ab_{ij}(l))}\to \HH^{(l,\ab_{i-1,k}(l))}.
\]
The definition of the maps $\tau_i^{(l,kj)}$ needs some preparation: for each $l=1,\ldots,n$ let $d_{l1}<d_{l2}<\cdots <d_{lr_l}$ be integers such that
\[
\{d_{l1},d_{l2},\cdots, d_{lr_l}\}=\{\ab_1(l),\ab_2(l),\cdots,\ab_m(l)\}.
\]
Now for each $l$ and each $k$ with $1<k$ we have $L_{l,d_{lk}}\subset L_{l,d_{l,k-1}}$ and choose a  homogeneous complex homomorphism
\[
\rho^{(l,d_k)}\: \HH^{(l,d_{lk})}\to \HH^{(l,d_{l,k-1})},
\]
which extends the inclusion map $L_{l,d_{lk}}\subset L_{l,d_{l,k-1}}$.

Now we come to the definition of $\tau_i^{(l,kj)}\: \HH^{(l,\ab_{ij}(l))}\to \HH^{(l,\ab_{i-1,k}(l))}$. We first notice that $\ab_{ij}=\ab_s$ and $\ab_{i-1,j}=\ab_t$ for some $s$ and $t$ with $1\leq s,t\leq m$, see Corollary~\ref{easily}.

\begin{Definition}
\label{tau}
{\em (i) We let  $\tau_i^{(l,kj)}=0$, if  $\lambda^{(i)}_{kj}=0$.

(ii) If  $\lambda^{(i)}_{kj}\neq 0$, then  $\ab_s(l)\geq \ab_t(l)$. If  $\ab_s(l)= \ab_t(l)$ we let  $\tau_i^{(l,kj)}=\id$.

 If $\ab_s(l)> \ab_t(l)$,  $\ab_s(l)=d_{lb}$ and $\ab_t(l)= d_{lc}$  for some $d_{lb}$ and $d_{lc}$ with $b>c$, then we let
\[
\tau_i^{(l,kj)}=\rho^{(ld_{lc})}\circ \rho^{(ld_{l,c-1})}\circ \cdots \circ \rho^{(ld_{lb})}.
\]
}
\end{Definition}
Notice that $\tau_i^{(l,kj)}\: \HH^{(l,\ab_{ij}(l))}\to \HH^{(l,\ab_{i-1,k}(l))}$ extends the inclusion map $L_{l,\ab_{ij}(l))}\subset L_{l,\ab_{i-1,k}(l)}$ whenever $\ab_{ij}(l)\geq \ab_{i-1,k}(l)$.

\medskip
We first show that $\sigma_i(\GG^{(i)})\subset \nn \GG^{(i-1)}$. For that it suffices to show that $\sigma_i^{(kj)}(\GG^{(ij)})\subset \nn \GG^{(i-1,k)}$. If $\lambda^{(i)}_{kj}=0$, then $\sigma_i^{(kj)}=0$. On the other hand, if $\lambda^{(i)}_{kj}\neq 0$, then from the definition (\ref{def}) of $\lambda^{(i)}_{kj}$ and the fact that the resolution of $I$ is minimal, it follows  that $\ab_{ij}>\ab_{i-1,k}$. Therefore, $\ab_{ij}(l)\geq \ab_{i-1,k}(l)$ for all $l$, and $\ab_{ij}(l)>\ab_{i-1,k}(l)$ for at least one $l$. Since $\HH^{(l,\ab_{ij}(l))}$ is an $\ab_{ij}(l)$-linear resolution and $\HH^{(l,\ab_{i-1,k}(l))}$ is an $\ab_{i-1,k}(l)$-linear resolution and since $\tau_i^{l,kj}$ is a homogeneous complex homomorphism, we see  that $\Im \tau_i^{(l, kj)}\subset \nn^{\ab_{ij}(l)-\ab_{i-1,k}(l)}\HH^{(l,\ab_{i-1,k}(l))}$. Hence
\[
\Im \sigma_i^{(kj)}\subset (\prod_{l=1}^n\nn^{\ab_{ij}(l)-\ab_{i-1,k}(l)}) \GG^{(i-1,k)}\subset \nn \GG^{(i-1,k)}.
\]

Next we prove that the complex homomorphism $\sigma_i$ extends the chain map $\partial^*\: F_i^*\to F_{i-1}^*$. In other words, we have to show that diagram~(\ref{diagram1}) commutes, where for all $i$ we denote by $\pi$ is the augmentation map of $G^{(i)}$.
\begin{eqnarray}
\label{diagram1}
\begin{CD}
G_0^{(i)}@> \sigma_i >> G_0^{(i-1)}\\
@V \pi V V @V V \pi V\\
F_i^*@>\partial^* >> F_{i-1}^*
\end{CD}
\end{eqnarray}
To prove this it suffices to show that the diagrams ~(\ref{diagram2}) commute, where the map   $L_{ij}\to L_{i-1,k}$  is the restriction of $\partial^*$ to the summands $L_{ij}$ and $L_{i-1,k}$.
\begin{eqnarray}
\label{diagram2}
\begin{CD}
\Tensor_{l=1}^n H_0^{(l,\ab_{ij}(l))}& = & G_0^{(ij)}@> \sigma^{(ij)} >> G_0^{(i-1,k)}& = &\Tensor_{l=1}^n H_0^{(l,\ab_{i-1,k}(l))}\\
&& @V \pi V V @V V \pi V &&\\
\prod_{l=1}^nL_{l,\ab_{ij}}(l)& = & L_{ij}@>>> L_{i-1,k}& = &\prod_{l=1}^nL_{l,\ab_{i-1k}}(l)
\end{CD}
\end{eqnarray}
If $\lambda^{(i)}_{kj}=0$, then $\sigma^{(ij)}=0$ and $L_{ij}\to L_{i-1,k}$ is the zero map, so that in this case the diagram commutes. Now let $\lambda^{(i)}_{kj}\neq 0$, and let  $h_1\tensor h_2\tensor \cdots \tensor h_n\in \Tensor_{l=1}^n H_0^{(l,\ab_{ij}(l))}$. Then
\[
\partial^*(\pi(h_1\tensor h_2\tensor \cdots \tensor h_n))=\lambda ^{(i)}_{kj}\pi(h_1)\pi(h_2)\cdots \pi(h_n).
\]
On the other hand,
\begin{eqnarray*}
\pi(\sigma^{(ij)}(h_1\tensor h_2\tensor\cdots \tensor h_n))&=& \pi(\lambda ^{(i)}_{kj}\tau_{i}^{(1,kj)}(h_{1})\tensor \cdots \tensor\tau_{i}^{(n,kj)}(h_{n}))\\
&=&\lambda ^{(i)}_{kj}\pi(\tau_{i}^{(1,kj)}(h_{1}))\tensor \cdots \tensor\pi(\tau_{i}^{(n,kj)}(h_{n}))\\
&=&\lambda ^{(i)}_{kj}\pi(h_1)\pi(h_2)\cdots \pi(h_n).
\end{eqnarray*}
The last equation follows  because  $\tau_{i}^{(l,kj)}\: \HH^{(l,\ab_{ij}(l))}\to \HH^{(l,\ab_{i-1,k}(l))}$ extends the inclusion map $L_{l,\ab_{ij}(l)}\subset L_{l,\ab_{i-1k}(l)}$.

\medskip

Finally we show that $\sigma_{i-1}\circ \sigma_i=0$. For that we need to show that $$\sum_{k=1}^{\beta_{i-1}}\sigma_{i-1}^{(qk)}\circ\sigma_i^{(kj)}=0$$ for all $q$ and $j$.

\medskip
We have

\[
\sigma_{i-1}^{(qk)}\circ\sigma_i^{(kj)}
=\lambda^{(i-1)}_{qk}(\tau_{i-1}^{(1,qk)}\tensor \tau_{i-1}^{(2,qk)}\tensor \cdots \tensor \tau_{i-1}^{(n,qk)})\circ \lambda^{(i)}_{kj}(\tau_{i}^{(1,kj)}\tensor \tau_{i}^{(2,kj)}\tensor \cdots \tensor \tau_{i}^{(n,kj)})
\]
\[
=\lambda^{(i-1)}_{qk}\lambda^{(i)}_{kj}(\tau_{i-1}^{(1,qk)}\circ\tau_i^{(1,kj)})\tensor (\tau_{i-1}^{(2,qk)}\circ\tau_i^{(2,kj)})\tensor \dots \tensor (\tau_{i-1}^{(n,qk)}\circ\tau_i^{(n,kj)}).
\]

We claim that the maps $\tau_{i-1}^{(l,qk)}\circ\tau_i^{(l,kj)}$ in this tensor product are equal to each other for all $k$ for which $\lambda^{(i-1)}_{qk}\lambda^{(i)}_{kj}\neq 0$.
Set $$\alpha^{(qj)}=(\tau_{i-1}^{(1,qk)}\circ\tau_i^{(1,kj)})\tensor (\tau_{i-1}^{(2,qk)}\circ\tau_i^{(2,kj)})\tensor \dots \tensor (\tau_{i-1}^{(n,qk)}\circ\tau_i^{(n,kj)}).$$
 Then the claim implies that the restriction of $\sigma_{i-1}\circ \sigma_i$ to $\GG^{(ij)}$  is given by
\[
(\sum_k \lambda^{(i-1)}_{qk}\lambda^{(i)}_{kj})\alpha^{(qj)}.
\]
This implies that $\sigma_{i-1}\circ \sigma_i=0$, as desired.

\medskip
It remains to prove the claim.  We will show that $\tau_{i-1}^{(l,qk)}\circ\tau_i^{(l,kj)}$ only depends on the numbers $\ab_{ij}(l)$ and $\ab_{i-2,q}(l)$ (and hence not on $k$). By Corollary~\ref{easily} there exist integers $s,t,u$ such  that $\ab_{ij}=\ab_s$, $\ab_{i-1,k}=\ab_t$ and $\ab_{i-2,q}=\ab_u$. By assumption, $\lambda^{(i-1)}_{qk}\lambda^{(i)}_{kj}\neq 0$. Therefore, $\lambda^{(i)}_{kj}\neq 0$ and  $\lambda^{(i-1)}_{qk}\neq 0$,  and hence  $\ab_s>\ab_t>\ab_u$. Therefore,  $\ab_s(l)\geq \ab_t(l)\geq \ab_u(l)$ for $l=1,\ldots,n$. If  $\ab_s(l)= \ab_t(l)= \ab_u(l)$ then $\tau_{i-1}^{(l,qk)}\circ\tau_i^{(l,kj)}=\id$, according to Definition~\ref{tau}.

If $\ab_s(l)> \ab_t(l)>\ab_u(l)$,  we have $\ab_s(l)=d_{ls}$, $\ab_t(l)= d_{lt}$ and $\ab_u(l)= d_{lu}$ with $s>t>u$. Then
\[
\tau_{i-1}^{(l,qk)}\circ\tau_i^{(l,kj)}=\rho^{(ld_{lu})}\circ \rho^{(ld_{l,u-1})}\circ \cdots \circ \rho^{(ld_{ls})},
\]
and hence $\tau_{i-1}^{(l,qk)}\circ\tau_i^{(l,kj)}$ depends only on $ld_{ls}=\ab_s(l)=\ab_ij{(l)}$ and on $d_{lu}=\ab_u(l)=\ab_{i-2,q}(l)$.

If $\ab_s(l)> \ab_t(l)=\ab_u(l)$, then $\tau_i^{(l,kj)}=\rho^{(ld_{lu})}\circ \rho^{(ld_{l,u-1})}\circ \cdots \circ \rho^{(ld_{lt})}$ and $\tau_{i-1}^{(l,qk)}=\id$, so that \[
\tau_{i-1}^{(l,qk)}\circ\tau_i^{(l,kj)}=\rho^{(ld_{lu})}\circ \rho^{(ld_{l,u-1)}}\circ \cdots \circ \rho^{(ld_{lt})}.
\]
Hence $\tau_{i-1}^{(l,qk)}\circ\tau_i^{(l,kj)}$ depends only on $ld_{ls}=\ab_s(l)=\ab_{ij}(l)$ and $d_{lt}=\ab_t(l)=\ab_u(l)=\ab_{i-2,q}(l)$.  Similarly one treats the case that  $\ab_s(l)= \ab_t(l)>\ab_u(l)$,

This completes the proof of the assertion that $\sigma^{(i-1)}\circ \sigma^{(i)}=0.$

\medskip

Now let $\DD$ be the double complex with $D_{ij}=G^{(j)}_i$ for all $i$ and $j$. The column of this double complex are the complexes $G^{(j)}$ and the row complexes are the complexes
\[
\to G_i^{(j)}\to G_{i-1}^{(j)}\to \cdots \to G_{0}^{(j)}
\]
with the $i$th component of the $\sigma_j$ as chain map.

\begin{Theorem}
\label{resolution}
The total complex $\Tot(\DD)$ of the double complex $\DD$ is a graded minimal free resolution of $T/L$.
\end{Theorem}

\begin{proof}
It follows from the construction of $\DD$ that $(\Tot(\DD), \delta)$ is a complex of graded free $T$-modules. Moreover, since each $\GG^{(j)}$ is a graded minimal free resolution of $F_j^*$ and since $\sigma_i (\GG^{(j)})\subset \nn \GG^{(j)}$ it follows that $\delta(\Tot(\DD))\subset \nn \Tot(\DD)$. Thus it remains to show that $\Tot(\DD)$ is acyclic with $H_0(\Tot(\DD))=T/L$. In order to prove this we compute the spectral sequence attached to $\DD$ with respect to the row filtration. Then $E^1_{ij}=H_i^h(D_{*,i})=H_i^h(G^{(j)})$. Since $G^{(j)}$ is a free resolution  of $F_j^*$ it follows that $E^1_{ij}=0$ for $i>0$ and $E^1_{0,j}=F_j^*$. Next we compute $E^2_{ij}=H^v_j(H_i^h(\DD))$, and see that $E^2_{ij}=0$ if $i>0$ and $E^2_{0j}=H_j(\FF^*)$.
It follows from  Theorem~\ref{main} that $E^2_{0j}=0$ for $i>0$ and $E^2_{00}=T/L$. This proves that $\Tot(\DD)$ is indeed acyclic with $H_0(\Tot(\DD))=T/L$.
\end{proof}

\begin{proof}[Proof of Theorem~\ref{regularity}]
Let $t_{k}(I)=\max\{j:\beta_{kj}(I)\neq 0\}$, then
\[
\reg(I)=\max\{t_{k}(I)-k:k=0,\dots,p\}.
\]
On the other hand,  $\GG^{(i)}= \Dirsum_{j=1}^{\beta_i}\GG^{(ij)}$, where $\GG^{(ij)}$ is an $|\ab_{ij}|$-linear resolution of $L_{ij}$. Therefore,
$t_{k}(F_{i}^*)=\max_{j}\{k+|\ab_{ij}|\}=k+t_{i-1}(I)$. Let $t_{k}(L)=\max\{j:\beta_{kj}(L)\neq 0\}$.
Then it follows from Theorem~\ref{resolution} that
\[
t_{k}(L)=\max\{t_{k}(F_{1}^*),t_{k-1}(F_{2}^*),\dots,t_{0}(F_{k+1}^*)\}
=\max\{k+t_{0}(I),(k-1)+t_{1}(I),\dots,t_{k}(I)\}.
\]
It follows that
\[
t_{k}(L)-k=\max\{t_{0}(I),t_{1}(I)-1,\dots,t_{k}(I)-k\}
\]
Hence
\[
\reg(L)=\max_k\{\max\{t_{0}(I),t_{1}(I)-1,\dots,t_{k}(I)-k\}\}=\max\{t_{0}(I),\dots,t_{p}(I)-p\}=\reg(I).
\]
\end{proof}

As another consequence of Theorem~\ref{resolution}, we obtain
\begin{Corollary}
With the notation introduced we have
\[
\projdim (T/L)=\max_{i,j}\{\sum_{l=1}^n\projdim(L_{l,\ab_{ij}(l)}) +i\}.
\]

\begin{proof}
By Theorem  ~\ref{resolution} the complex $\Tot(\DD)$ is a minimal free resolution of $L$. Therefore,
\[
\projdim (T/L)=\max_{i,j}\{\sum_{l=1}^n\projdim F_{i}^*+i:i=0,\ldots,p\}.
\]
On the other hand, since $F^*_i=\Dirsum_{j=1}^{\beta_i} L_{ij}$, we conclude that $\projdim F_{i}^*$ is the maximum number
among the numbers $\projdim L_{ij}$. Hence
\[
\projdim (F_i^*)=\max_{i,j}\{\projdim L_{ij}\}
\]
Now by formula (\ref{product}) we have
\[
\projdim L_{ij}=\sum_{l=1}^n \projdim L_{l, \ab_{ij}(l)},
\]
and hence
\[
\projdim (T/L)=\max_{i,j}\{\sum_{l=1}^n\projdim L_{l,\ab_{ij}(l)} +i\}.
\]
\end{proof}

\end{Corollary}


\begin{thebibliography}{99}

\bibitem{BH} S. Bayati, J. Herzog, Expansion of monomial ideals and multigraded modules, Rocky Mountain J. Math., to appear.

\bibitem{CN}
A. Conca and E. De Negri. M-sequences, graph ideals, and ladder ideals of linear type. J. Algebra,\textbf{211} (1999), 599-624.

\bibitem{HT}
T. Hoa, N. Tam, On some invariants of a mixed product idals, Arch. Math. \textbf{94} (2010), 327-337.

\bibitem{IR}
C. Ionescu and G. Rinaldo, Some algebraic invariants related to mixed product ideals, Arch. Math.  \textbf{91} (2008), 20-30.

\bibitem{RV}
G. Restuccia and R. Villarreal, On the normality of monomial ideals of mixed products, Commun. Algebra, \textbf{29} (2001), 3571-3580.

\bibitem{R}
G. Rinaldo, Betti numbers of mixed product ideals, Arch. Math.  \textbf{91} (2008), 416-426.

\end{thebibliography}
\end{document}